\documentclass[12pt]{article}

\usepackage{amsmath} 
\usepackage{amsthm}
\usepackage{amssymb}
\usepackage{enumitem} 
\usepackage{mathtools}
\usepackage[hidelinks]{hyperref}
\usepackage[utf8]{inputenc}
\usepackage{csquotes}
\usepackage[pagewise]{lineno}
\usepackage[english]{babel}
\usepackage{enumitem}
\usepackage{titlesec}
\usepackage{graphicx}
\usepackage{subcaption}
\usepackage{mwe}
\usepackage{bm}
\usepackage[title]{appendix}
\usepackage{tikz-cd}
\usepackage{caption}

\usepackage{array}  
\usepackage{multirow}  
\usepackage{geometry}  
\usepackage{diagbox}  

\usepackage{geometry}
\makeatletter
\BeforeBeginEnvironment{proof}{%
  \def\@Opargbegintheorem#1#2#3#4{#4\trivlist
      \item[]{#3#2\@thmcounterend\ }}%
}
\AfterEndEnvironment{proof}{%
  \def\@Opargbegintheorem#1#2#3#4{#4\trivlist
      \item[\hskip\labelsep{#3#1}]{#3(#2)\@thmcounterend\ }}%
}
 \newtheorem{thm}{Theorem}[subsection]
 \newtheorem{cor}[thm]{Corollary}
 \newtheorem{lem}[thm]{Lemma}
 \newtheorem{prop}[thm]{Proposition}
 \newtheorem{defn}{Definition}[subsection]
 \newtheorem{exmp}{Example}[subsection]
 \newtheorem{rema}{Remark}[subsection]

\providecommand{\keywords}[1]
{
	\small	
	\textbf{\text{Keywords:}} #1
}

\providecommand{\subclass}[1]
{
	\small	
	\textbf{\text{MSC(2010):}} #1
}

\title{\Large An Algebraic Generalization of the Ramanujan Sum}

\author{\normalsize N. Uday Kiran\thanks{nudaykiran@sssihl.edu.in}  \\
	\begin{small}Department of Mathematics and Computer Science\end{small}\\
	\small Sri Sathya Sai Institute of Higher Learning, Puttaparthi, India.
}
\date{}

\begin{document}
\maketitle

\begin{center}        
        \textit{In honor of George Andrews and Bruce Berndt on the occasion of their 85th birthday.}
\end{center}

\vspace{0.1cm}
    
\begin{abstract}
Ramanujan sums have attracted significant attention in both mathematical and engineering disciplines due to their diverse applications. In this paper, we introduce an algebraic generalization of Ramanujan sums, derived through polynomial remaindering. This generalization is motivated by its applications in Restricted Partition Theory and Coding Theory. Our investigation focuses on the properties of these sums and expresses them as finite trigonometric sums subject to a coprime condition. Interestingly, these finite trigonometric sums with a coprime condition, which arise naturally in our context, were recently introduced as an analogue of Ramanujan sums by Berndt, Kim, and Zahaescu. Furthermore, we provide an explicit formula for the size of Levenshtein codes with an additional parity condition (also known as Shifted Varshamov-Tenengolts deletion correction codes), which have found many interesting applications in studying the Little-Offord problem, DNA-based data storage and distributed synchronization. Specifically, we present an explicit formula for a particularly important open case $\text{SVT}_{t,b}(s \pm \delta, 2s + 1)$ for $s$ or $s+1$ are divisible by $4$ and for small values of $\delta$.
\\
\\
		\\
		\subclass{1L03 $\cdot$  11H71 $\cdot$ 11P81}
\end{abstract}

\section{Introduction}
Using the notation in \cite{toth}, we denote a polynomial
\[
R_{k}(q) = \sum_{t=0}^{k-1} c_{k}(t) q^{t},
\]
where $c_{k}(t)$ is the Ramanujan (trigonometric) sum for $t$ an integer. In this paper, our motivation for a generalization of the Ramanujan sum is the elegant congruence relation \cite{uk}:
\begin{equation}\label{poly_rem_RS}
R_{k}(q) \equiv (q)_{k-1} \mod 1-q^{k},    
\end{equation}
where $(q)_{k}=(1-q)\cdots (1-q^{k})$ is the \textit{restricted $q$-product} for a positive integer $k$, and $(q)_{0}=1$ is defined by convention. These products frequently arise in the study of $q$-series and partition functions. 

Since $R_{k}(q)$ is a polynomial of degree $(k-1)$, it is indeed the remainder when $(q)_{k-1}$ is divided by $1-q^{k}$. The equation (\ref{poly_rem_RS}) also provides a combinatorial interpretation:
\begin{equation}\label{RS_Comb}
c_{k}(t) = \left(\begin{matrix}
    \text{number of even-sized} \\
    \text{subsets of } \mathcal{A}_{k} \text{ whose} \\
    \text{elements sum to } t \text{ mod } k
\end{matrix}\right) - \left(\begin{matrix}
    \text{number of odd-sized} \\
    \text{subsets of } \mathcal{A}_{k} \text{ whose} \\
    \text{elements sum to } t \text{ mod } k
\end{matrix}\right),
\end{equation}
where $\mathcal{A}_{k} = \{1, 2, \ldots, k-1\}$, and by convention, we consider the empty set to be even-sized. Surprisingly, although several generalizations of the Ramanujan sum have been considered in the literature, none of them are based on this combinatorial interpretation. 

In this paper, we present a combinatorial-based generalization of the Ramanujan sum using polynomial remaindering. Our approach focuses on $\mathcal{A}_{s}$ for various values of $s$ within the range $1 \leq s \leq k$. We denote this generalization as $\sigma^{(b)}_{k}(t;s)$, introducing two additional parameters: $s$ and $b$. A notable feature of our generalization is its linear recurrence property with respect to both $t$ and $s$.

The generalization \(\sigma^{(b)}_{k}(t;s)\) essentially consists of two types of finite trigonometric sums: one is a direct sum, while the other is subject to a coprime condition (similar to the one given in \cite{bruce} and the references therein). For \(1 \leq s \leq k\),
\[
\sigma^{(0)}_{k}(t;s) = \frac{1}{k}\sum_{j=0}^{k-1}\alpha^{-jt}(\alpha^{j})_{s-1} \quad \quad \text{ and } \quad \quad \sigma^{(1)}_{k}(t;s) = \frac{1}{k}\sum_{\substack{j=0 \\ (j,k)=1}}^{k-1}\alpha^{-jt}(\alpha^{j})_{s-1},
\]
where \(\alpha = e^{2\pi i/k}\). Thus, \(\sigma^{(1)}_{k}(t;s)\) has an additional coprime condition compared to \(\sigma^{(0)}_{k}(t;s)\). Furthermore, from the observation that \((\alpha^{j})_{k-1}=k\) if \(\gcd(j,k)=1\) and \(0\) otherwise, we have 
$
\sigma_{k}^{(b)}(t;k)=c_{k}(t),
$ that is, the Ramanujan sum is recovered in the case $s=k$.

Recently, in \cite{uk}, we observed that the values \(\sigma^{(1)}_{k}(t;s)\) serve as coefficients within the framework of Sylvester Wave Theory as applied to restricted partitions. In this work, we explore this generalization further and present a relevant formula in coding theory. It is well-known that Ramanujan sums play a crucial role in determining the size of a specific class of deletion codes known as Levenshtein codes \cite{bibak}. These codes are quite versatile, capable of correcting various types of errors, including deletions, insertions, bit reversals, and transpositions of adjacent bits (see \cite{bibak} and the references therein). Despite their usefulness in numerous applications such as the Little-Offord problem, DNA-based data storage and distributed synchronization, many problems related to these codes remain open. 

For positive integers $k,s$ with $1\leq s\leq k$, Levenshtein codes $L_{t}(s,k)$ is a set of binary codes $(b_{s-1},\cdots,b_{1})$ that correct single-bit deletions, defined by 
$$
\sum^{s-1}_{j=1} j b_{j} \equiv t \text{ mod } k
$$
Clearly, this congruence relation is associated with the distinct partitions over the set \(\mathcal{A}_{s}\) given above. Levenshtein introduced the concept of edit distance, showing that \(L_{t}(s, 2s + 1)\) can correct single deletion, insertion, transposition, and substitution errors. Bibak et al. \cite{bibak1} provided an explicit formula for \(|L_{t}(s, 2s + 1)|\), the cardinality of \(L_{t}(s, 2s + 1)\), in terms of the Ramanujan sums.

Levenshtein codes with the parity condition (also known as Shifted Varshamov-Tenengolts codes) $\text{SVT}_{t,b}(s;k)$ are a subset of $L_{t}(s,k)$ that include an additional parity condition, i.e., 
$$
\sum^{s-1}_{j=1} j b_{j} \equiv t \text{ mod } k \quad \text{ and } \quad \sum_{j=1}^{s-1} b_{j} \equiv b \text{ mod } 2.
$$
From this definition, we have $|L_{t}(s, 2s + 1)| = |\text{SVT}_{t,0}(s,k)| + |\text{SVT}_{t,1}(s,k)|$. The reason these codes are called ``shifted" is that they can correct a single deletion when the location of the deleted bit is known to be within certain consecutive positions. Bibak et al. \cite{bibak} provided a formula using trigonometric functions. In this work, we present an explicit formula for the significant case $\text{SVT}_{t,b}(s \pm \delta, 2s + 1)$ for small values of $\delta$ such as $0, 1, 2, 3$, with $s \equiv 0 \text{ or } 3 \mod 4$. The results for this case are unknown in the literature, and determining the scenario when $\delta > 0$ is particularly challenging. The elegance of our approach lies in deriving explicit formulas for these difficult cases using only linear recurrence and the quadratic Gauss sum, which is a novel contribution to this literature.

The remainder of the paper is organized into two sections. In Section \ref{sec_gen}, we discuss the definition, general properties, and a combinatorial interpretation of our generalization, along with the associated trigonometric sums. In Section \ref{sec_SVT}, we derive the size of $\text{SVT}_{t,b}(s, 2s + 1)$ for the cases where $4 \mid s$ or $4 \mid (s + 1)$.

\section{Our Generalization and its Properties}\label{sec_gen}

Let $(q)_{k}=(1-q)\cdots (1-q^{k})$ and $(q)_{0}=1$ be the restricted $q$-products. Our parametric generalization of the Ramanujan sum is based on polynomial remaindering.

\begin{defn}\label{defn_GRS}[Generalization of the Ramanujan Sum]
Let $b$ and $t$ be two non-negative integers, $s$ and $k$ be positive integers such that $1\leq s\leq k$. The generalized Ramanujan sum $\sigma^{(b)}_{k}(t;s)$ is given as the coefficients of the polynomial
$$
R^{(b)}_{k,s}(q) = \sum^{k-1}_{t=0} \sigma_{k}^{(b)}(t;s)q^{t} = \frac{1}{k^{b}}(q)^{b}_{k-1}(q)_{s-1} \pmod{1-q^{k}}.
$$
As \( R_{k,s}^{(b)}(q) \) is a polynomial of degree \( (k-1) \), it is, in fact, the remainder when divided by \( 1 - q^k \).
\end{defn}

As discussed in the Introduction, the parameter \( s \) extends the combinatorial meaning of the Ramanujan sum. In fact, \( s = k \) corresponds to the Ramanujan sum case, that is, \( \sigma_{k}^{(b)}(t;k) = c_{k}(t) \). Furthermore, the parameter \( b \) is crucial for generating two distinct types of finite trigonometric sums, particularly for small values of \( s \).

Similar to the Ramanujan sum, the sequence of values \( \sigma_{k}^{(b)}(t;s) \) with respect to \( t \) is a periodic sequence with period \( k \). Specifically, for \( t \geq 1 \), we have
\[
\sigma_{k}^{(b)}(t;s) = \sigma_{k}^{(b)}(t \% k; s),
\]
where \( \% \) denotes the remainder operator. A key feature of these sums is that they satisfy a generalized Pascal-like linear recurrence relation.

\begin{thm}[Linear Recurrence]\label{linear_recurrence}
For $k\geq 1$ and $1\leq s<k$ we have 
$$
\sigma_{k}^{(b)}(t;s+1)=\sigma_{k}^{(b)}(t;s)-\sigma_{k}^{(b)}(t-s;s).
$$    
\end{thm}
\begin{proof} Follows from observing that 
$
(q)_{s}=(q)_{s-1}-q^{s}(q)_{s-1}. 
$
\end{proof}

Using the above recurrence relation, one can perform efficient computations of the values with the initial data (Proposition \ref{init_value}):  
$$
\sigma^{(0)}_{k}(t;1)= \left\{\begin{matrix}
    1 & \textnormal{ for } t=0 \\
    0 & \textnormal{elsewhere}. 
\end{matrix}\right.
$$
For $b\geq 1$,
$$
\sigma^{(b)}_{k}(t;1)=\frac{1}{k}c_{k}(t).
$$
The time required to compute all values \(\sigma^{(b)}_{k}(t;s)\) will be linear with respect to the total number of values. Using the linear recurrence, we can easily obtain the values of \(\sigma_{k}^{(b)}(t;s)\) for specific cases. For instance, we have the following result.

\begin{prop} For $p$ prime we have 
    $$
    \sigma_{p}^{(b)}(t;p-1)= t-\frac{p+1}{2} \textnormal{ for } 1\leq t\leq p. 
    $$
\end{prop}
\begin{proof} By linear recurrence and 
$$
    \sigma_{p}^{(b)}(p;p)=c_{p}(p)=p-1 \textnormal{ and } \sigma_{p}^{(b)}(t;p)=c_{p}(t)=-1 \textnormal{ for } 1\leq t\leq p-1.
    $$
\end{proof}

Ramanujan sum has the following beautiful arithmetical structure:
$$
c_{k}(t)=\sum_{d|(k,t)}\mu\left(\frac{k}{d}\right)d.
$$
Using the above linear recurrence, one can understand the arithmetical structure of $\sigma^{(1)}_{k}(t;s)$. Indeed, we have 
$$
\sigma^{(1)}_{k}(t;2)=\frac{1}{k}\left(c_{k}(t)-c_{k}(t-1)\right) 
$$
and 
$$
\sigma^{(1)}_{k}(t;3)=\frac{1}{k}\left(c_{k}(t)-c_{k}(t-1)-c_{k}(t-2)+c_{k}(t-3)\right).
$$
For \( k > 2 \), we can consolidate the above equations for \( s \in \{1,2,3\} \) and \( s \leq t \leq k + s - 1 \) as follows:
\[
\sigma_{k}^{(1)}(t; s) = \frac{1}{k} \sum_{\delta = 0}^{s-1} \sum_{d \mid (k, t - \delta)} (-1)^{\textnormal{wt}(\delta)} \mu \left( \frac{k}{d} \right) d,
\]
where \( \textnormal{wt}(\delta) \) denotes the number of 1s in the binary representation of \( \delta \). This expression shows that the value of \( \sigma_{k}^{(1)}(t; s) \) depends on the greatest common divisor of $k$ and the translates of \( t \).

\begin{table}
    \centering
    \begin{tabular}{|c|c|c|c|c|c|c|c|} 
        \hline
        \diagbox{$s$}{$t$} & $0$ & $1$ & $2$ & $3$ & $4$ & $5$   \\
        \hline
        $1$ & $1$ & $0$ & $0$ & $0$ & $0$ & $0$   \\
        \hline
        $2$ & $1$ & $-1$ & $0$ & $0$ & $0$ & $0$   \\
        \hline
        $3$ & $1$ & $-1$ & $-1$ & $1$ & $0$ & $0$   \\
        \hline
        $4$ & $0$ & $-1$ & $-1$ & $0$ & $1$ & $1$   \\
        \hline
        $5$ & $1$ & $-1$ & $-2$ & $-1$ & $1$ & $2$   \\
        \hline
        $6$ & $2$ & $1$ & $-1$ & $-2$ & $-1$ & $1$   \\
        \hline
    \end{tabular}
    \hspace{1cm}  
    \begin{tabular}{|c|c|c|c|c|c|c|c|}
        \hline
       \diagbox{$s$}{$t$}  & $0$ & $1$ & $2$ & $3$ & $4$ & $5$   \\
        \hline
          $1$ & $\frac{1}{3}$ & $\frac{1}{6}$ & $-\frac{1}{6}$ & $-\frac{1}{3}$ & $-\frac{1}{6}$ & $\frac{1}{6}$   \\
        \hline
        $2$ & $\frac{1}{6}$ & $-\frac{1}{6}$ & $-\frac{1}{3}$ & $-\frac{1}{6}$ & $\frac{1}{6}$ & $\frac{1}{3}$   \\
        \hline
        $3$ & $0$ & $-\frac{1}{2}$ & $-\frac{1}{2}$ & $0$ & $\frac{1}{2}$ & $\frac{1}{2}$   \\
        \hline
        $4$ & $0$ & $-1$ & $-1$ & $0$ & $1$ & $1$   \\
        \hline
        $5$ & $1$ & $-1$ & $-2$ & $-1$ & $1$ & $2$   \\
        \hline
        $6$ & $2$ & $1$ & $-1$ & $-2$ & $-1$ & $1$   \\
        \hline
    \end{tabular}
 
        \caption{The table on the left displays \( \sigma_{6}^{(0)}(t; s) \), while the table on the right displays \( \sigma_{6}^{(b)}(t; s) \) for \( b \geq 1 \).}

    \label{table_6}
\end{table}

Table \ref{table_6} provides the values for the case \( k = 6 \). From the table, it can be observed that the values \( \sigma^{(b)}_{k}(t; s) \) are either integers or rational numbers of the form \( \frac{a}{k} \), where \( a \) is an integer and \( k \) is the given base. Moreover, after a certain \( s \), the rows in both tables are equal. The last row corresponds to the Ramanujan sum. These observations can be made more precise for the general case using the following lemma.

\begin{lem} \label{lem_freq_1}
For $b\geq 1$, we have 
\begin{equation}\label{sigma_freq_1}
\sigma^{(b)}_{k}(t;s) = \frac{1}{k}\sum_{\xi\in \Delta_{k}} (\xi)_{s-1}\xi^{-t},     
\end{equation}
where $\Delta_{k}=\{\alpha^{h}:\alpha = e^{2\pi i/k}, (h,k)=1 \text{ and } 1\leq h\leq k\}$ is the collection all primitive $k^{th}$ roots of unity.     
\end{lem}

\begin{proof}
 Since $\left\{\sigma_{k}^{(b)}(t;s)\right\}_{t=0}^{\infty}$ is a periodic sequence, the generating function
$$
F(q)=\sum_{t\geq 0} \sigma_{k}^{(b)}(t;s) q^{t} = R_{k,s}^{(b)}(q)\left(\sum_{t=0}^{\infty}q^{kt}\right),
$$
has a finite Fourier series expansion given by 
\begin{equation}\label{fourier_series_F}
F(q) = \sum_{t\geq 0} \sigma_{k}^{(b)}(t;s) q^{t} = \sum_{t\geq 0} \left(\sum_{j=0}^{k-1} a(j)\xi^{jt}\right) q^{t},    
\end{equation}
where $a(j)=\frac{1}{k}R_{k,s}^{(b)}(\xi^{-j})$ and $\xi=e^{2\pi i/k}$. Comparing the coefficients of any fixed $q^{t}$ gives us the finite Fourier series for $\sigma_{k}^{(b)}(t;s)$, namely 
$$
\sigma_{k}^{(b)}(t;s) = \frac{1}{k}\sum_{j=0}^{k-1}R_{k,s}^{(b)}(\xi^{-j})\xi^{jt}.
$$
As $\xi^{-j}$, for $0\leq j<k$, is a $k^{th}$ root of unity, using  
$$
(\xi^{-j})_{k-1} = (1-\xi^{-j})(1-\xi^{-2j})\cdots (1-\xi^{-j(k-1)})= \left\{\begin{matrix} 
k & \textnormal{ for } (j,k)=1\\
0 & \textnormal { otherwise },
\end{matrix}\right. 
$$
we obtain
$$
a(j)=\frac{1}{k}R_{k,s}^{(b)}(\xi^{j})=\left\{\begin{matrix} 
\frac{1}{k}(\xi^{-j})_{s-1} & \textnormal{ for } (j,k)=1\\
0 & \textnormal { otherwise }. 
\end{matrix}\right.
$$
Since $\Delta_{k}$ is defined only for $(j,k)=1$ and by (\ref{fourier_series_F}) we have the result. 
\end{proof}

\begin{rema} 
Our generalization clearly extends the Ramanujan sum in a direct and insightful manner. For \( s = k \) and \( b \geq 0 \), we have
\[
c_{k}(t) = \sigma_{k}^{(b)}(t; k) = \frac{1}{k} \sum_{\xi \in \Delta_{k}} (\xi)_{k-1} \xi^{-t} = \sum_{\xi \in \Delta_{k}} \xi^{-t},
\]
which precisely matches the definition originally provided by Ramanujan \cite{Rama}.
\end{rema}

We can compute the initial values with the above formula. 

\begin{cor} \label{init_value}
For $k,b\geq 1$ and $0\leq t<k$ we have 
$$
    \sigma_{k}^{(b)}(t;1)=\frac{1}{k}\sigma_{k}^{(b)}(t;k)=\frac{1}{k}c_{k}(t).
$$
\end{cor}
\begin{proof} Using $(\xi)_{k-1}=k$ for $\xi\in \Delta_{k}$,
$$
\sigma_{k}^{(b)}(t;k)=\frac{1}{k}\sum_{\xi\in \Delta_{k}}(\xi)_{k-1}\xi^{-t}=c_{k}(t)=k\left(\frac{1}{k}\sum_{\xi\in \Delta_{k}}\xi^{-t}\right)=k\sigma^{(b)}_{k}(t;1).
$$    
\end{proof}

Notice that the right hand side of (\ref{sigma_freq_1}) is independent of $b$. Therefore, it suffices to consider the values of $\sigma^{(b)}_{k}(t;s)$ only for the two cases $b=0,1$.
\begin{cor}
For $b\geq 1$, $\sigma_{k}^{(b+1)}(t;s)=\sigma_{k}^{(b)}(t;s)$.    
\end{cor}

Following similar steps of the proof of Lemma \ref{lem_freq_1} and by the finite Fourier series, \(\sigma^{(0)}_{k}(t;s)\) is given by
\begin{equation}
\sigma^{(0)}_{k}(t;s) = \frac{1}{k} \sum_{j=0}^{k-1} (\alpha^j)_{s-1} \alpha^{-jt},
\end{equation}
where \(\alpha = e^{2\pi i / k}\). Notice that there is no coprime condition. However, it is easy to show that the sum $\sigma_{k}^{(0)}(t;s)$ could have a coprime condition for higher values of $s$. 

\begin{lem} 
Let $s \geq k/p$, where $p$ is the smallest prime divisor of $k$. Then,  
\begin{equation}\label{sigma_freq_0}
\sigma^{(0)}_{k}(t; s) = \frac{1}{k} \sum_{\substack{j=0 \\ (j, k) = 1}}^{k-1} \alpha^j \alpha^{-jt} = \frac{1}{k} \sum_{\xi \in \Delta_{k}} (\xi)_{s-1} \xi^{-t}, 
\end{equation}
where $\alpha = e^{2\pi i / k}$ and $\Delta_{k}$ is the collection of primitive $k$-th roots of unity. Consequently,
\begin{equation}
\sigma^{(1)}_{k}(t; s) = \sigma^{(0)}_{k}(t; s) \quad \text{for} \quad s \geq k/p,
\end{equation}
and, thus, $\sigma^{(1)}_{k}(t; s)$ is an integer.
\end{lem}

\begin{proof}
From the case $s\geq k/p$ we can see that the set $\{k-s,\cdots,k-1\}$ does not contain any divisor of $k$. Thus, the factors $(1-x^{k-s-1}),\cdots, (1-x^{k-1})$ do not vanish for any $x=\alpha^{j}, j=1,2,\cdots,(k-1)$. Therefore, we have $R_{k,s}^{(0)}(\alpha^{j})=(\alpha^{-j})_{s-1}$ for $(j,k)=1$ and $R_{k,s}^{(0)}(\alpha^{j})=0$ for $(j,k)\neq 1$. The result follows. 
\end{proof} 

\begin{cor}
For $b\geq 0$ and $s\geq k/p$ where $p$ is the smallest prime dividing $k$ we have 
$$
R_{k,s}^{(b)}(q)=(q)_{s-1} \mod 1-q^{k}.
$$
Consequently, for this case the value $\sigma_{k}^{(b)}(t;s)$ is an integer. 
\end{cor}

\begin{thm}[Backward Linear Recurrence]\label{backward_linear_rec}

For \(s \geq 1\),
$$
\sigma_{k}^{(b)}(t;s) = -\frac{1}{k} \sum_{j=1}^{k-1} j \sigma^{(b)}_{k}(t - js; s + 1)
$$
\end{thm}
\begin{proof}
Let \(\xi\) be a \(k^{\text{th}}\) root of unity. Then, we have
$$
(1 - \xi)(\xi + 2\xi^{2} + \cdots + (k - 1)\xi^{k-1}) = -k.
$$
Consider
\begin{eqnarray}
\nonumber \sigma^{(1)}_{k}(t;s) &=& \frac{1}{k} \sum_{\xi \in \Delta_{k}} (\xi)_{s-1} \xi^{-t} \frac{1 - \xi^{s}}{1 - \xi^{s}} = -\frac{1}{k^{2}} \sum_{\xi \in \Delta_{k}} (\xi)_{s} \xi^{-t} \sum_{j=1}^{k-1} j \xi^{js} \\
\nonumber &=& -\frac{1}{k} \sum_{j=1}^{k-1} j \sigma^{(1)}_{k}(t - js; s + 1).
\end{eqnarray}
Similarly, we can obtain an expression for \(\sigma^{(0)}_{k}(t;s)\).
\end{proof}

\subsection{Finite Trigonometric Sums}

Our generalization introduces finite trigonometric sums subject to a coprime condition, which were recently proposed as analogues to the Ramanujan sum in \cite{bruce}. This extension not only enriches the theoretical framework but also opens new avenues for practical applications in related fields.

We have two types of sums: when \( b = 0 \), the sums are direct sums, while for \( b = 1 \), the sums are analogues of the Ramanujan sum under a coprime condition. The two sums are 
$$
\sigma^{(0)}_{k}(t;s) = \frac{1}{k}\sum_{j=0}^{k-1}\alpha^{-jt}(\alpha^{j})_{s-1} \quad\quad\text{ and }\quad\quad \sigma^{(1)}_{k}(t;s) = \frac{1}{k}\sum_{\stackrel{j=0}{(j,k)=1}}^{k-1}\alpha^{-jt}(\alpha^{j})_{s-1},
$$
where $\alpha = e^{2\pi i/k}$. By Lemma \ref{lem_freq_1} we can write
\begin{eqnarray}
\nonumber \sum_{d|k}\frac{d}{k}\sigma^{(1)}_{d}(t;s) = \frac{1}{k}\sum_{d|k}\sum_{\xi\in \Delta_{d}}\xi^{-t}(\xi)_{s-1} = \frac{1}{k}\sum_{j=0}^{k-1}\alpha^{-jt}(\alpha^{j})_{s-1}=\sigma_{k}^{(0)}(t;s),
\end{eqnarray}
where $\alpha$ is a primitive $k^{th}$ root of unity. Thus, the two types of sums are connected with the Dirichlet convolution.

\begin{prop} 
$$
    k\sigma_{k}^{(0)}(t;s)=\sum_{d|k}d\sigma^{(1)}_{d}(t;s)
$$
where we assume that $\sigma^{(1)}_{d}(t;s)=0$ if $s>d$. That is, $k\sigma_{k}^{(0)}=I*(k\sigma^{(1)}_{k})$, where $*$ is the Dirichlet convolution and $I(n)=1$ constant function. By Möbius inversion we have
\[
k \sigma_{k}^{(1)}(t; s) = \sum_{d \mid k} \mu\left(\frac{k}{d}\right) d \sigma_{d}^{(0)}(t; s).
\]
 
\end{prop}

As one would expect, the analogues of the Ramanujan sums considered in \cite{bruce} are indeed related to the Ramanujan sum. We will first explain the idea through the simple case $\sigma_{k}^{(b)}(0;2)$. Consider 
$$
\sigma^{(0)}_{k}(0;2) = \frac{1}{k}\sum_{j=0}^{k-1}(1-\xi^{j}),
$$
where $\xi=e^{2\pi i/k}$. Taking the factor $\xi^{l/2}$ out, multiplying and dividing by $2i$ we get 
$$
\sigma^{(0)}_{k}(0;2) = -2i\frac{1}{k}\sum_{j=0}^{k-1}\left(\frac{\xi^{j/2}-\xi^{-j/2}}{2i}\right)\xi^{j/2}.
$$
Using the formulae $\sin \theta=\frac{e^{i\theta}-e^{-i\theta}}{2i}$ and $e^{i\theta}-\cos \theta + i \sin \theta$, we have 
$$
\sigma^{(0)}_{k}(0;2) = -2i\frac{1}{k}\sum_{j=0}^{k-1}\sin \frac{\pi j}{k}(\cos \frac{\pi j}{k}+i\sin \frac{\pi j}{k}).
$$
Comparing the real parts we have the identity
$$
\sigma^{(0)}_{k}(0;2) = \frac{2}{k}\sum_{j=0}^{k-1} \sin^{2}\left(\frac{\pi j}{k}\right).
$$
Similarly, we have  
$$
\sigma^{(1)}_{k}(0;2) = \frac{2}{k}\sum_{\stackrel{j=0}{(j,k)=1}}^{k-1} \sin^{2}\left(\frac{\pi j}{k}\right).
$$
Further, the values of $\sigma_{k}^{(b)}(0;2)$ can be easily determined: 
$$
\sigma^{(1)}_{k}(0;2)=\frac{1}{k}\sum^{k}_{\stackrel{j=1}{(j,k)=1}}(1-\xi^{j})=\frac{1}{k}\sum^{k}_{\stackrel{j=1}{(j,k)=1}}\xi^{0}-\frac{1}{k}\sum^{k}_{\stackrel{j=1}{(j,k)=1}}\xi^{j}=\frac{c_{k}(k)-c_{k}(k-1)}{k},
$$
where $\xi=e^{2\pi i/k}$.

\begin{thm} For $k\geq 1$ we have the following identities: 
$$
\sum_{j=0}^{k-1} \sin^{2}\left(\frac{\pi j}{k}\right)=\frac{k}{2}
$$
and 
$$
\sum_{\stackrel{j=0}{(j,k)=1}}^{k-1} \sin^{2}\left(\frac{\pi j}{k}\right) = \frac{c_{k}(k)-c_{k}(k-1)}{2}.
$$ 
\end{thm}

Now, consider $\sigma_{k}^{(b)}(0;3)$. Taking a similar approach as before we have 
$$
\sigma_{k}^{(0)}(0;3)=\frac{1}{k}\sum_{j=0}^{k-1}(1-\xi^{j})(1-\xi^{2j}),
$$
where $\xi=e^{2\pi i/k}$. By pulling out the factors $\xi^{j/2}$ and $\xi^{2j/2}$ and multiplying and diving by $2i$ we have 
$$
\sigma_{k}^{(0)}(0;3)=-\frac{4}{k}\sum_{j=0}^{k-1}\frac{(\xi^{j/2}-\xi^{-j/2})}{2i}\frac{(\xi^{j}-\xi^{-j})}{2i}\xi^{3j/2}.
$$
Again, using the trigonometric formulae and comparing the real and imaginary parts, we have the following:

\begin{thm} For $k>2$ we have the following results:

\begin{enumerate}
    \item $$
\sum_{l=0}^{k-1}\sin\left(\frac{\pi l}{k}\right)\sin\left(\frac{2\pi l}{k}\right)\cos\left(\frac{3\pi l}{k}\right) = \left\{\begin{matrix}
    -\frac{k}{4} & k\geq 4 \\
    -\frac{3}{2} & k=3
\end{matrix}\right.
$$

\item 
$$
\sum_{\stackrel{1\leq l\leq k}{(l,k)=1}}\sin\left(\frac{\pi l}{k}\right)\sin\left(\frac{2\pi l}{k}\right)\cos\left(\frac{3\pi l}{k}\right) =\frac{c_{k}(k-1)+c_{k}(k-2)-c_{k}(k)-c_{k}(k-3)}{4}
$$

\item 
$$
\sum_{l=0}^{k-1}\sin\left(\frac{\pi l}{k}\right)\sin\left(\frac{2\pi l}{k}\right)\sin\left(\frac{3\pi l}{k}\right) =0
$$
\item 
$$
\sum_{\stackrel{1\leq l\leq k}{(l,k)=1}}\sin\left(\frac{\pi l}{k}\right)\sin\left(\frac{2\pi l}{k}\right)\sin\left(\frac{3\pi l}{k}\right)=0.
$$
\end{enumerate}
  
\end{thm}

We now consider the case $s$ divisible by $4$ and $k\geq s(s+1)/4$. 
$$
\sigma^{(1)}_{k}(k-s(s+1)/4;s) = \frac{1}{k}\sum_{\xi \in \Delta_{k}} (1-\xi)(1-\xi^{2})\cdots (1-\xi^{s})\xi^{-(k-s(s+1)/4)}.
$$
Taking a factor $\xi^{j/2}$ from each of the $j^{th}$ factor, multiplying and dividing by $2i$ we get 

$$
\sigma^{(0)}_{k}\left(k-\frac{s(s+1)}{4};s+1\right)=2^{s}\sum_{h=1}^{k} \sin\frac{h\pi}{k}\cdots \sin \frac{sh\pi}{k}
$$

$$
\sigma^{(1)}_{k}\left(k-\frac{s(s+1)}{4};s+1\right)=2^{s}\sum_{\stackrel{h=1}{(h,k)=1}}^{k} \sin\frac{h\pi}{k}\cdots \sin \frac{sh\pi}{k}
$$

\begin{thm} Let $k$ and $s$ be two positive integers such that $s$ is divisible by $4$ and $k/p\leq s\leq k$ for $p$ the smallest prime divisor of $k$. Then, the following equation holds: 
$$
\sum_{\stackrel{h=1}{(h,k)=1}}^{k} \sin\frac{h\pi}{k}\cdots \sin \frac{sh\pi}{k}=\sum_{h=1}^{k} \sin\frac{h\pi}{k}\cdots \sin \frac{sh\pi}{k} = \frac{1}{2^{s}}\sigma_{k}^{(b)}\left(k-\frac{s(s+1)}{4};s+1\right).
$$    
\end{thm}
Suppose $s=4$ then from the above equation we have
$$
\frac{1}{2^{4}}\sigma_{k}^{(b)}\left(k-5;5\right)=\sum_{\stackrel{h=1}{(h,k)=1}}^{k} \sin\frac{h\pi}{k}\sin \frac{2h\pi}{k}\sin \frac{3h\pi}{k}\sin \frac{4h\pi}{k}.
$$
For sufficiently large $k$, that is, $k$ larger than $30$, these values turn out to be zero. 

\subsection{Combinatorial Interpretation}

As stated in the Introduction, our algebraic generalization of the Ramanujan sum in Definition \ref{defn_GRS} directly reflects combinatorial significance through the coefficients of the restricted \( q \)-product. Suppose that we write the $q$-product in the form 
$$
(q)_{s-1}=(1-q)\cdots (1-q^{s-1}) = \sum_{m=0}^{s(s-1)/2} a_{m}q^{m}.
$$
Then, the coefficient $a_{m}$ can be interpreted as the difference in the number of even-size partitions $(E_{s}(m))$ and odd-size partitions $(O_{s}(m))$ of $m$ considering a distinct part of the set $\{1,2,\cdots,s-1\}$, namely $a_{m}=E_{s}(m)-O_{s}(m)$.

The reminder of a polynomial with $1-q^{k}$ can also be treated as sieving of the coefficients of a polynomial. In particular, we have 

$$
R_{k,s}^{(0)}(q)=\sum_{t=0}^{k-1} \sigma_{k}^{(0)}(t;s)q^{t}=(q)_{s-1} \mod 1-q^{k}. 
$$
Therefore,
$$
\sigma_{k}^{(0)}(t;s)=\sum_{j=0}^{\lfloor (k-1)/2 \rfloor} a_{jk+t}=\sum_{j=0}^{\lfloor (k-1)/2 \rfloor} \left(E_{s}(jk+t)-O_{s}(jk+t)\right).
$$
The case $s=k$ corresponding to the Ramanujan sum is well known \cite{Ramanathan}. 

\begin{cor} \label{comb_0}
For $k>0$ and $1\leq s\leq k$, we can write
\begin{equation}\label{GRS_sgn}
\sigma^{(0)}_{k}(t;s)=\sum_{\textnormal{sum}_{k}(A)= t} (-1)^{|A|},\quad \textnormal{ for } A\subseteq \{1,2,\cdots,s-1\},
\end{equation}
where $|A|$ is the cardinality of $A$ and $\textnormal{sum}_{k}(A)$ stands for the sum modulo $k$ of the elements of $A$, with the convention $\textnormal{sum}_{k}(\emptyset)=0$.
\end{cor}

\begin{proof}
The largest sum possible with distinct parts from $\{1,2,\cdots,s-1\}$ is $s(s-1)/2$. By the combinatorial interpretation given above and iterating over various sums considering the number of even-size and odd-size sums we have:
$$
\sigma_{k}^{(0)}(t;s)=\sum^{s(s-1)/2}_{\stackrel{l=0}{l\% k=t}} \left(E_{s}(l)-O_{s}(l)\right)=\sum_{\stackrel{\textnormal{sum}_{k}(A)= t}{|A| \textnormal{ even}}} 1\ -\sum_{\stackrel{\textnormal{sum}_{k}(B)= t}{|B| \textnormal{ odd}}} 1,
$$
where $A,B\subseteq \{1,2,\cdots,s-1\}$.
\end{proof}

The combinatorial interpretation of $\sigma^{(1)}_{k}(t;s)$ is a bit cumbersome but it is also straightforward. By the sieving of coefficients, for an arbitrary $s\in \{0,1,\cdots,k-1\}$, we have the following interpretation from the definition:
$$
\sigma^{(1)}_{k}(t;s)=\frac{1}{k} \sum_{j=0}^{\lfloor\frac{k-1}{2}\rfloor}\left(\tilde{E}_{k,s}(jk+t)-\tilde{O}_{k,s}(jk+t)\right).
$$
Here $\tilde{E}_{k,s}(l)$ is the number of partitions of $l$ into an even number of parts each less than $k$ and permitting parts of size $< s$ at most twice and others at most once. Similarly, $\tilde{O}_{k,s}(l)$ is the number of partitions of $l$ into an odd number of parts each less than $k$ and permitting parts of size $< s$ at most twice and others at most once.

From the integrality of \( \sigma_{k}^{(1)}(t; s) \), for \( s \geq \frac{k}{p} \), where \( p \) is the smallest prime divisor of \( k \), we have demonstrated that the difference in parity is a multiple of \( k \) for sufficiently large \( s \).

From a combinatorial perspective, our generalization accurately captures the parity difference among distinct partitions. This is particularly evident when \( s \) divides \( k \). If \( s \) divides \( k \), then \( 1 - q^s \) divides \( 1 - q^k \). Consequently,
\[
R_{s}(q) = R_{k,s}^{(0)}(q) \mod 1 - q^s.
\]
By sieving the indices of the coefficients of \( R_{k,s}(q) \) with \( s \) and comparing terms from the above equation, we obtain the following result:

\begin{prop}
For positive integers \( k \) and \( s \) with \( s \mid k \), we have
\[
c_{s}(t) = \sum_{j=0}^{k/s-1} \sigma_{k}^{(0)}(js + t; s).
\]
\end{prop}

This equation demonstrates that the generalized Ramanujan sum effectively captures the parity difference among distinct partitions. Specifically, the parity difference for \( t \mod s \) is equivalent to the sum of the parity differences for \( t + js \mod k \), where \(0\leq j <k/s.\) 

\begin{prop}
For positive integers \( k \) and \( s \) with \( s \mid k \), we have
\[
 \sum_{j=0}^{k/s-1} \sigma_{k}^{(1)}(js + t; s)=0.
\]    
\end{prop}

\begin{figure*}[htbp]
    \centering
    \begin{subfigure}[b]{0.45\textwidth}
        \centering
        \includegraphics[width=\textwidth]{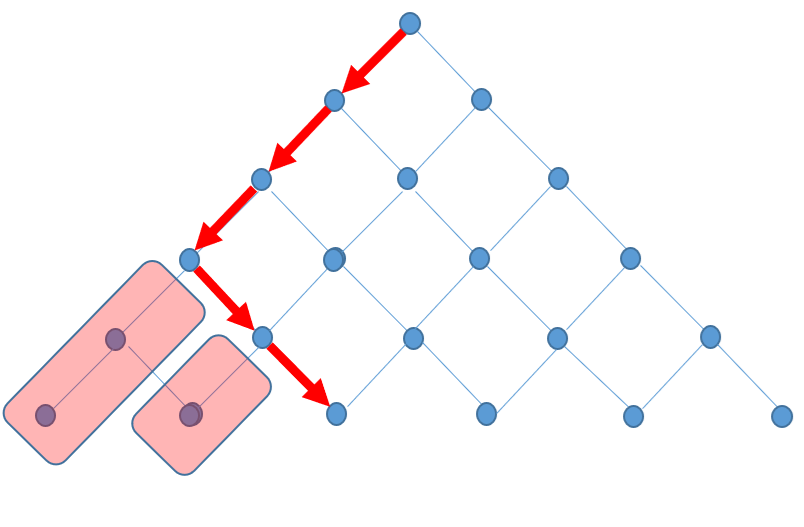}
        \caption*{\small Even number of parts: $2+1$}
        \label{fig:image1}
    \end{subfigure}%
    \hspace{0.05\textwidth} 
    \begin{subfigure}[b]{0.45\textwidth}
        \centering
        \includegraphics[width=\textwidth]{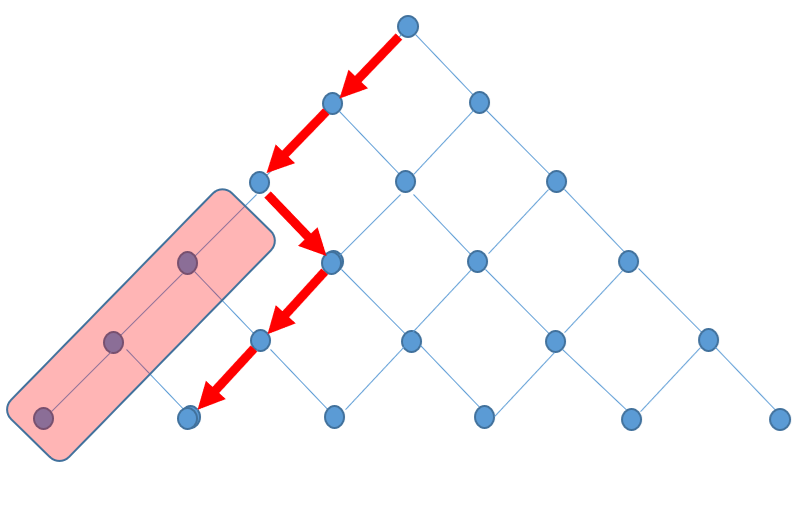}
        \caption*{\small Odd number of parts: $3$}
        \label{fig:image2}
    \end{subfigure}

    \vspace{1em} 

    \begin{subfigure}[b]{0.45\textwidth}
        \centering
        \includegraphics[width=\textwidth]{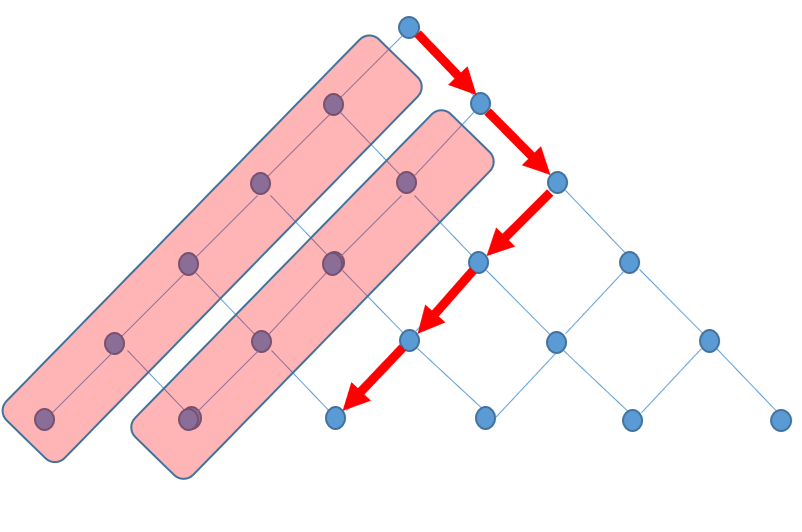}
        \caption*{\small Even number of parts: $5+4$}
        \label{fig:image3}
    \end{subfigure}%
    \hspace{0.05\textwidth} 
    \begin{subfigure}[b]{0.45\textwidth}
        \centering
        \includegraphics[width=\textwidth]{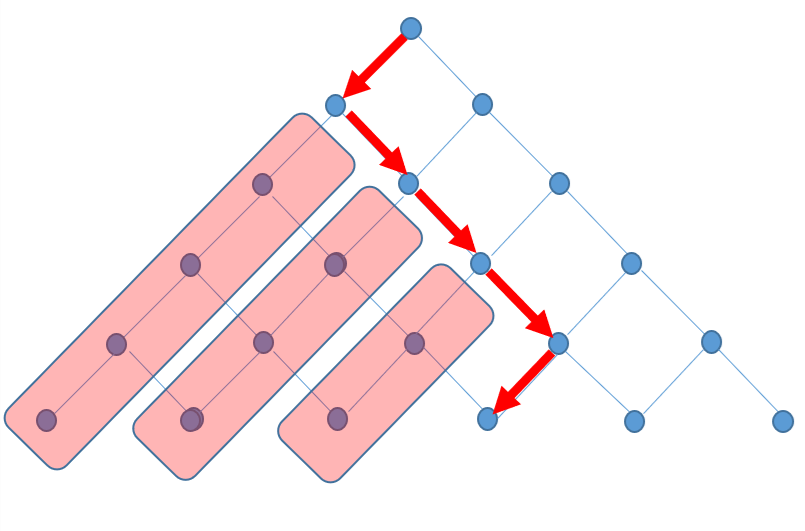}
        \caption*{\small Odd number of parts: $4+3+2$}
        \label{fig:image4}
    \end{subfigure}

    \vspace{1em} 

    \begin{subfigure}[b]{0.45\textwidth}
        \centering
        \includegraphics[width=\textwidth]{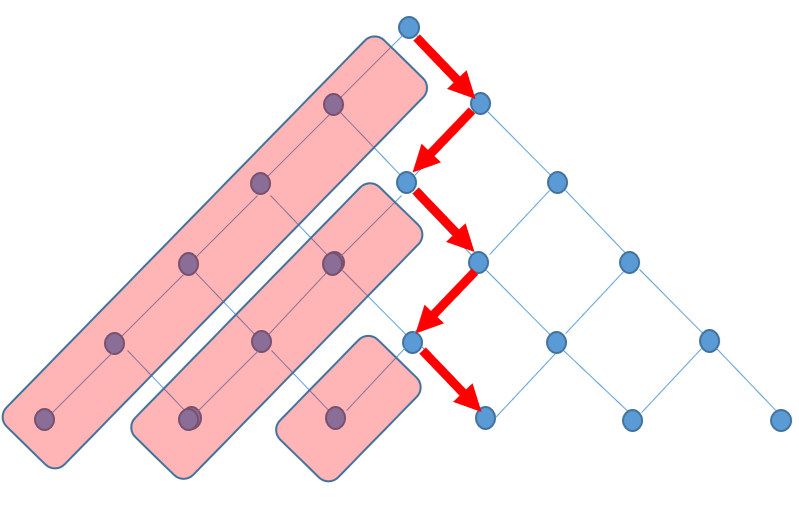}
        \caption*{\small Odd number of parts: $5+3+1$}
        \label{fig:image5}
    \end{subfigure}%
    \hspace{0.05\textwidth} 
    \begin{subfigure}[b]{0.45\textwidth}
        \centering
        \includegraphics[width=\textwidth]{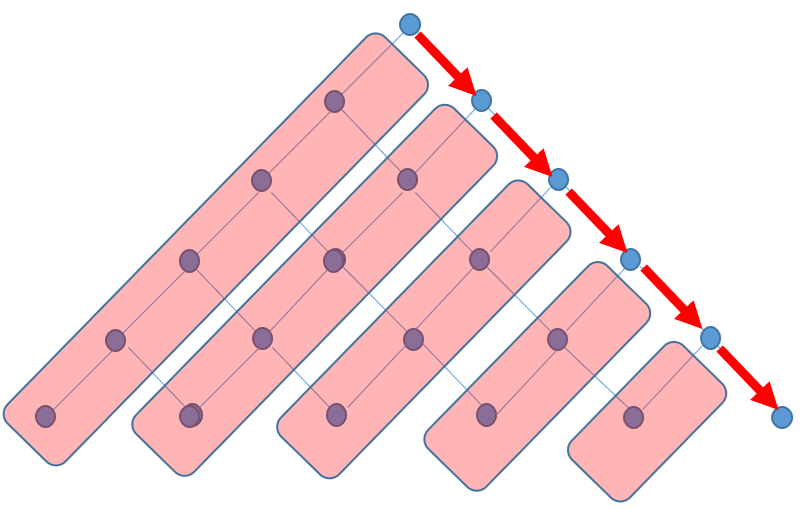}
        \caption*{\small Odd number of parts: $5+4+3+2+1$}
        \label{fig:image6}
    \end{subfigure}

    \caption{\small Case \( k = 6 \) and \( s = 6 \): \( c_{6}(3) = \# \text{even partitions} - \# \text{odd partitions} = 2 - 4 = -2 \). Alternatively, consider \( k = 12 \) and \( s = 6 \) to compute $\sigma_{12}(3;6)$ and $\sigma_{12}(9;6)$ from equations (\ref{eq_1}) and (\ref{eq_2}) to obtain \( c_{6}(t) = \sigma_{12}(3;6) + \sigma_{12}(9;6) \). }
    \label{fig:grid}
\end{figure*}

\begin{exmp} 
Let $k=12$ and $s=6$. Then, we have 
\begin{eqnarray}
\nonumber R_{12,6}(x)&=&(1-x)(1-x^2)(1-x^3)(1-x^4)(1-x^5)\mod (1-x^{12})\\   
&=& -x^{10} - x^9 - x^8 + x^7 + x^6 + x^5 - x^3 + 1 \\ 
\nonumber &=& \sum_{t=0}^{11} \sigma_{12}(t;6)x^{t}.\label{exmp_7}
\end{eqnarray}
Consider the case $t=3$. Then, we have the sets $\{1,2\}, \{3\},\{4,5\}, \{2,3,4\}, \{1,3,5\}$ and $\{1,2,3,4,5\}$ satisfying the property $\textnormal{sum}_{12}(A)=t$. These subsets can also be visualized as paths in a Pascal-Lattice (see Figure \ref{fig:grid}). One can also represent these paths as binary vectors \cite{uk2}. From (Corollary \ref{comb_0}) we have
\begin{eqnarray}
\sigma_{12}(3;6)&=&(-1)^{|\{1,2\}|}+(-1)^{|\{3\}|} + (-1)^{|\{1,2,3,4,5\}|}= 1-2 = -1,\label{eq_1}\\
\sigma_{12}(9;6)&=&(-1)^{|\{4,5\}|}+(-1)^{|\{2,3,4\}|}+(-1)^{|\{1,3,5\}|} =  1-2 = -1.\label{eq_2}
\end{eqnarray}
We could also the case as $k=6$ and $s=6$ in order to compute $c_{6}(t)$. In this case, 
$$
R_{6,6}(x)= x^5 - x^4 - 2x^3 - x^2 + x + 2=\sum^{5}_{t=0}c_{6}(t)x^{t}.
$$
We have 
\begin{eqnarray}
\nonumber c_{6}(3) &=& \sigma_{12}(3;6)+\sigma_{12}(9;6) = -1-1 = -2.
\end{eqnarray}

\end{exmp}

The above decomposition of Ramanujan sum can be extended to obtain natural decomposition of arithmetic functions based on their expansion. Suppose an arithmetic function is expanded as 
$$
f(n)=\sum_{s=1}^{\infty} a_{s} c_{s}(t),
$$
where the series converges absolutely. Let $k=sm$. Then, we can write 
$$
f(n)=\sum_{s=1}^{\infty} a_{s} c_{s}(n)=\sum_{s=0}^{\infty}a_{s}\left(\sum_{j=0}^{m-1}\sigma_{ms}^{(0)}(js+n;s)\right)=\sum_{j=0}^{m-1}\sum_{s=0}^{\infty}a_{s}\sigma_{ms}^{(0)}(js+n;s).
$$
Defining the functions 
$$
f_{j}(n)=\sum_{s=1}^{\infty}a_{s}\sigma_{ms}^{(0)}(js+n;s),
$$
for $j=0,\cdots, m-1$. We have the decomposition $f(n)=\sum^{m-1}_{j=0}f_{j}(n).$ A question that naturally arises in this context is: What is the arithmetic function defined by  
$$
g_{r,j}(n)=\sum_{s=1}^{\infty}\frac{1}{s^{r}}\sigma_{ms}^{(0)}(js+n;s)
$$
for $r >0$? By the well-known Dirichlet series expansion we have
$$
\sum_{j=0}^{m-1} g_{r,j}(n) = \sum_{s=1}^{\infty}\frac{1}{s^{r}}c_{s}(n)=\frac{\sigma_{r-1}(n)}{n^{r-1}\zeta(r)},
$$
where $\sigma_{r}$ is the divisor function. 

We also have an explicit expression for the following function: 
$$
f_{\alpha,s}(t)=\sum_{k=1}^{\infty}\frac{\sigma^{(1)}_{k}(t;s)}{k^{\alpha}}, 
$$
for $\alpha>0$ a positive integer. Suppose $(1-x)_{s-1}=\sum_{j=0}^{s(s-1)/2}a_{j}x^{j}$. Then, we have 
$$
\sigma^{(1)}_{k}(t;s)=\frac{1}{k}\sum^{s(s-1)/2}_{j=0}a_{j}c_{k}(t-j),
$$
where $a_{j}$ is the difference of even parity distinct partitions of $j$ and odd parity distinct partitions of $j$ with parts from $\{1,2,\cdots,s-1\}.$  Substituting for $\sigma_{k}^{(1)}(t,s)$ we have 
\begin{eqnarray}
 \nonumber    f_{\alpha,s}(t)&=&\sum_{k=1}^{\infty}\frac{\sigma^{(1)}_{k}(t;s)}{k^{\alpha}}=\sum_{j=0}^{s(s-1)/2}a_{j}\sum_{k=1}^{\infty}\frac{c_{k}(t-j)}{k^{\alpha+1}} \\
 \nonumber &=& \left(\frac{1}{\zeta(\alpha+1)}\sum_{\stackrel{j=0}{t\neq j}}^{s(s-1)/2}\frac{a_{j}\sigma_{\alpha}(|t-j|)}{|t-j|^{\alpha}}\right) + \frac{6}{\pi^{2}}\zeta(\alpha +1).
\end{eqnarray}

\section{Size of $\text{SVT}_{t,b}(s,2s+1)$, for $4|s$ or $4|(s+1)$}\label{sec_SVT}

In this section we show that our algebraic generalization of the Ramanujan sum is helpful in the context of deletion correction codes. We will look at an application associated with the Levishtein codes with a parity condition (or the Shifted Varshamov-Tenengolts codes) used for deletion correction (see \cite{bibak, bibak1} and references therein).   

A coefficient of $(q)_{s-1}=(1-q)\cdots(1-q^{s-1})$ is determined by which terms in each factor are considered. This can be represented in a binary code form based on the notion of position sum. Let $\mathbf{x}=(b_{1},\cdots,b_{s-1})\in \mathbb{F}^{s-1}_{2}$. Then, define the position sum as 
$$
\textnormal{PS}(\mathbf{x})=\sum_{j=1}^{s-1}jb_{j}.
$$
The Hamming weight of a code is defined in the usual way, $\textnormal{wt}(\mathbf{x})=\sum^{s-1}_{j=1}b_{j}$. 

\begin{defn}\label{SVT}
Let $k,s$ be positive integers, $t\in \mathbb{Z}_{k}$, and $r\in \{0,1\}$. The Shifted Varshamov-Tenengolts code $\textnormal{SVT}_{t,r}(s,k)$ is the set of all binary $s$-tuples $\mathbf{x}=(b_{1},\cdots,b_{s-1})$ such that 
$$
\textnormal{PS}(\mathbf{x})=\sum_{j=1}^{s-1}jb_{j}\equiv t\mod k,\quad \textnormal{wt}(\mathbf{x})=\sum_{j=1}^{s-1}b_{j} \equiv r \mod 2.
$$ 
\end{defn}

\begin{cor} For $k>0$ and $1< s\leq k$ we have 
$$
\sigma_{k}^{(0)}(t;s)=|\textnormal{SVT}_{t,0}(s,k)|-|\textnormal{SVT}_{t,1}(s,k)|=\sum_{\textnormal{PS}(\mathbf{x})\%k\equiv t}(-1)^{\textnormal{wt}(\mathbf{x})},\quad  \mathbf{x}\in \mathbb{F}_{2}^{s-1}.
$$   
\end{cor}

We will determine the size of $\textnormal{SVT}_{t,r}(s;2s+1)$, where $s$ or $s+1$ is divisible by $4$. 

\begin{thm}[\cite{bibak1}, Theorem 4.5]\label{sum}

\begin{eqnarray}
 \nonumber |\text{SVT}_{t,1}(s,2s+1)|&+&|\text{SVT}_{t,0}(s,2s+1)|  =  \\
 \nonumber && \frac{1}{2s+1}\sum_{d|2s+1}(-1)^{\frac{d-\textbf{I}_{4|d-1}}{4}}2^{\frac{2s+1-d}{2d}}c_{d}\left(\frac{1}{16}(d-\textbf{I}_{4|d-1})(3d+1)-t\right),
\end{eqnarray}
where $\textbf{I}_{4|d-1}$ is $1$ if $4$ divides $d-1$ and $-1$ if $4$ doesn't divide $d-1$.
\end{thm}

In view of the above theorems, if one wishes to determine the size of $\textnormal{SVT}_{t,r}(s;2s+1)$ we just need to determine $\sigma^{(0)}_{2s+1}(t;s+1)$.

\begin{thm} \label{difference}
For a positive number $s$ such that $4|s$ or $4|(s+1)$, and $k=2s+1$ we have 
$$
|\textnormal{SVT}_{t,0}(s,2s+1)|-|\textnormal{SVT}_{t,1}(s,2s+1)|=\sigma^{(0)}_{k}(t;s+1)=\frac{1}{k}\sum_{j=0}^{k-1}c_{k}\left(j^{2}+\frac{s(s+1)}{4}-t\right).
$$
If $k$ is a perfect square, then 
$$
|\textnormal{SVT}_{t,0}(s,2s+1)|-|\textnormal{SVT}_{t,1}(s,2s+1)|=\sigma_{k}^{(0)}(t;s+1)=\frac{1}{\sqrt{k}}c_{k}\left(\frac{s(s+1)}{4}-t\right).
$$
\end{thm}

\begin{proof}
Given the conditions, $k \mod 4$ is either $1$ or $3$. Define 
\[
\varepsilon_k = \begin{cases}
1 & \text{if } k \equiv 1 \pmod{4}, \\
i & \text{if } k \equiv 3 \pmod{4}.
\end{cases}
\]

For $\xi$ a primitive $k^{th}$ root of unity, that is $\xi\in \Delta_{k}$, we have $(\xi)_{k-1}=k$. So, from $k=2s+1$ we can write  
\begin{eqnarray}
 \nonumber k=(1-\xi)\cdots(1-\xi^{k-1})&=&(1-\xi)\cdots(1-\xi^{s})(1-\xi^{s+1})\cdots(1-\xi^{2s})\\
 \nonumber &=&(1-\xi)\cdots(1-\xi^{s})(1-\xi^{-s})\cdots(1-\xi^{-1})\\
  &=& (-1)^{s}\xi^{-s(s+1)/2}(1-\xi)^{2}\cdots(1-\xi^{s})^{2}.\label{roots_simplify}
\end{eqnarray}
By the quadratic Gauss sum (see for instance \cite{bressoud}) we have 
$$
G(\xi)=\sum_{j=0}^{k-1}\xi^{j^{2}}=\varepsilon_{k}\left(\frac{a}{k}\right)\sqrt{k},
$$
where $\left(\frac{a}{k}\right)$ is the Jacobi symbol for $a$ being in the exponent of $ \xi = \exp\left(2\pi i a/k \right)$.

By the identity $G(\xi)^{2}=(-1)^{s}k$, we can express the right hand side of (\ref{roots_simplify}) as  
$$
\xi^{-s(s+1)/2}(1-\xi)^{2}\cdots(1-\xi^{s})^{2}= (-1)^{s}k=G(\xi)^{2}.
$$
Taking square root both sides we have 
\begin{equation}\label{quantumness}
\xi^{-s(s+1)/4}(1-\xi)\cdots(1-\xi^{s})=\varepsilon_{k}\left(\frac{a}{k}\right)\sqrt{k}=G(\xi)=\sum_{j=0}^{k-1}\xi^{j^{2}}.
\end{equation}
Thus 
$$
(1-\xi)\cdots(1-\xi^{s})=\sum_{j=0}^{k-1}\xi^{j^{2}+s(s+1)/4}.
$$
Substituting the above expression into the formulae for $\sigma_{k}^{(0)}(t;s+1)$ we have 
\begin{eqnarray}
\nonumber \sigma^{(0)}_{k}(t;s+1)&=&\frac{1}{k}\sum_{\xi\in \Delta_{k}}\xi^{-t}(1-\xi)\cdots(1-\xi^{s})\\
\nonumber &=& \frac{1}{k}\sum_{\xi\in \Delta}\sum_{j=0}^{k-1}\xi^{-t}\xi^{j^{2}+s(s+1)/4}= \frac{1}{k}\sum_{j=0}^{k-1}c_{k}(j^{2}+s(s+1)/4-t).
\end{eqnarray}
Additionally,  
$$
\sigma_{k}^{(0)}(t;s+1)=\frac{\varepsilon_k}{\sqrt{k}}\sum_{\stackrel{a=1}{(a,k)=1}}^{k} \left(\frac{a}{k}\right) \textnormal{exp}\left(\frac{2\pi ia\left(\frac{s(s+1)}{2}-t\right)}{k}\right).
$$

Now consider the special case where $k$ is a perfect square. In this case, we necessarily have $k\equiv 1 \mod 4$, and therefore $\varepsilon_{k}=1$. Moreover, $\left(\frac{a}{k}\right)=1$ for $(a,k)=1$. Thus, from (\ref{quantumness}) we have
\begin{eqnarray}
\nonumber \sigma_{k}^{(0)}(t;s+1) &=& \frac{1}{k} \sum_{\xi \in \Delta_{k}} \xi^{-t} (1 - \xi) \cdots (1 - \xi^{s}) \\
\nonumber &=& \frac{1}{k} \sum_{\xi \in \Delta_{k}} \sqrt{k} \xi^{-t + s(s+1)/4} = \frac{1}{\sqrt{k}} c_{k}(s(s+1)/4 - t).
\end{eqnarray}

\end{proof}

\begin{rema}
Thanks to the quadratic Gauss sum, we were able to not only determine the correct sign for square root of $k$ in the above proof, but also express the formula in terms of Ramanujan sums, which are integers.  Ignoring the sign determined by the Jacobi symbol, arising from the Gauss sum, can lead to incorrect results. To illustrate this point, consider the case $k = 17, s = 8$ and $t=1$. Using the key step (\ref{quantumness}) in the above proof to obtain $\sigma_{17}^{(0)}(1;9)$ we have 
$$
\xi^{-18}(1-\xi)\cdots (1-\xi^{8})=\sum_{j=0}^{16} \xi^{j^2}=\left(\frac{a}{17}\right)\sqrt{17},
$$
where $\xi=\textnormal{exp}(2\pi i a /17),$ for $1\leq a\leq 16$. Thus  
$$
\sigma_{17}^{(0)}(1;9)=\frac{1}{17}\sum_{\stackrel{a=1}{(a,17)=1}}^{16} \left(\frac{a}{17}\right)\sqrt{17}=0,
$$
which agrees with the correct value. Instead of using (\ref{quantumness}), had we simply taken the positive square root $\sqrt{17}$, that is, if we consider wrongly: $
\xi^{-18}(1-\xi)\cdots (1-\xi^{8})=\sqrt{17},
$
we would have obtained
\begin{equation}
\sigma_{17}^{(0)}(1; 9) = \frac{1}{\sqrt{17}} c_{17}(0)=\frac{16}{\sqrt{17}},
\end{equation}
an irrational quantity, which is clearly absurd as the problem at hand involves counting and must yield a non-negative integer. This leads us to an important lesson: in calculations involving square roots, careful attention must be paid to the choice of sign.
\end{rema}

From theorems \ref{sum} and \ref{difference}, we can obtain both an explicit and an efficiently computable formula for the size of \( \textnormal{SVT}_{t,r}(s, 2s+1) \) when \( s \) or $s+1$ is divisible by 4. Furthermore, using the recurrence relation theorems \ref{linear_recurrence} and \ref{backward_linear_rec}, we can also provide direct formulas for \( \sigma^{(0)}_{k}(t;s\pm\delta)=\textnormal{SVT}_{t,0}(s\pm\delta, 2s+1)-\textnormal{SVT}_{t,1}(s\pm\delta, 2s+1) \) for $\delta=1,2,3$ which are otherwise very difficult to determine. In order to compute $|\textnormal{SVT}_{t,b}(s\pm\delta, 2s+1)|$ we define 
$$
\eta_{k}(t;s)=\frac{1}{k}\sum_{j=0}^{k-1}(1+\alpha^{j})\cdots (1+\alpha^{(s-1)j})\alpha^{-jt}
$$
where $\alpha = e^{\frac{2\pi i}{k}}$. It can be easily shown that 
$$
\eta_{k}(t;s+1)=|\text{SVT}_{t,1}(s,k)|+|\text{SVT}_{t,0}(s,k)|.
$$
Thus, given by Theorem \ref{sum} for the case of our interest. Moreover, similar to $\sigma_{k}^{(0)}(t;s)$, we have the linear recurrence relation
$$
\eta_{k}(t;s+1)=\eta_{k}(t;s)+\eta_{k}(t-s;s).
$$
and the following backward linear recurrence. 

\begin{thm}For $s>0$ and $k$ is an odd positive integer,
$$
\eta_{k}(t;s)=\frac{1}{2}\sum_{j=0}^{k-1}(-1)^{j}\eta_{k}(t-js;s+1).
$$    
\end{thm}
\begin{proof} The result follows from the observation
 $$
 (1+\xi)(1-\xi+\xi^{2}-\cdots+\xi^{k-1})=2.
 $$   
\end{proof}
In conclusion, for $\delta\geq 0$, we can write 
$$
|\text{SVT}_{t,b}(s\pm\delta,2s+1)| = \frac{1}{2}\left(\eta_{2s+1}(t;s+1\pm\delta)+(-1)^{b}\sigma_{2s+1}^{(0)}(t;s+1\pm\delta)\right).
$$




\section*{Acknowledgment}
This work is supported by the SERB-MATRICS project (MTR/2023/000705) from the Department of Science and Technology, India.




\end{document}